\documentclass[11pt]{amsart}
\usepackage{amsmath,amsthm}
\usepackage{amssymb,latexsym}
\usepackage{enumerate}

\newtheorem{thm}{Theorem}[section]
\newtheorem{lem}[thm]{Lemma}
\newtheorem{cla}[thm]{Claim}

\theoremstyle{definition}

\numberwithin{equation}{section}

\makeatletter
\@namedef{subjclassname@2010}{%
  \textup{2010} Mathematics Subject Classification}
\makeatother

\newcommand{\Leg}[2]{\left(\frac{#1}{#2}\right)}
\newcommand{\br}[1]{\left(#1\right)}

\newcommand{\Q}{\mathbb{Q}}
\newcommand{\Z}{\mathbb{Z}}

\newcommand{\af}{\mathfrak{a}}

\newcommand{\Oc}{\mathcal{O}}
\newcommand{\Rf}{\mathfrak{R}}
\newcommand{\Tf}{\mathfrak{T}}
\newcommand{\Lf}{\mathfrak{L}}
\begin{document}

%%%%% To ease editing, for IMPAN journals add:

\baselineskip=17pt

%%%%%%%%%%%%%%%%
\title[Class number one problem]{The class number one problem for the real quadratic fields $\Q\left(\sqrt{(an)^2+4a}\right)$}

\author[A. Bir\'o]{Andr\'as Bir\'o}
\address{A. R\'enyi Institute of Mathematics, Hungarian Academy of Sciences\\
1053 Budapest, Re\'altanoda u. 13-15, Hungary}
\email{biro.andras@renyi.mta.hu}

\author[K. Lapkova]{Kostadinka Lapkova}
\address{A. R\'enyi Institute of Mathematics, Hungarian Academy of Sciences\\
1053 Budapest, Re\'altanoda u. 13-15, Hungary}
\email{lapkova.kostadinka@renyi.mta.hu}

% \author{Andr\'as Bir\'o, \,\, Kostadinka Lapkova
% \footnote{The first author is partially supported by the Hungarian National Foundation for Scientific Research (OTKA) Grants  no. K100291, K104183, K109789 and ERC-AdG. Grant no. 321104. The second author is supported by Back-to-Research Grant of University of Vienna and partially supported by OTKA no. K104183.} \\
% \textit{A. R\'enyi Institute of Mathematics, Hungarian Academy of Sciences}\\
% \textit{1053 Budapest, Re\'altanoda u. 13-15, Hungary}\\
% \textit{biro.andras@renyi.mta.hu \,\,  lapkova.kostadinka@renyi.mta.hu}
%}

\date{}

%\maketitle
% \begin{center}
% \textit{A. R\'enyi Institute of Mathematics, Hungarian Academy of Sciences\\
% 1053 Budapest, Re\'altanoda u. 13-15, Hungary\\
% E-mails: biro.andras@renyi.mta.hu \,\,  lapkova.kostadinka@renyi.mta.hu}\end{center}

%% Classification and key words; note that the 2010 classification is used:

% \renewcommand{\thefootnote}{}
% \footnote{2010 \emph{Mathematics Subject Classification}: Primary 11R11; Secondary 11R29\,, 11R42.}
% \footnote{\emph{Key words and phrases}:  class number problem, real quadratic field.}
% \renewcommand{\thefootnote}{\arabic{footnote}}
% \setcounter{footnote}{0}
% \begin{center}
% \textit{A. R\'enyi Institute of Mathematics, Hungarian Academy of Sciences\\
% 1053 Budapest, Re\'altanoda u. 13-15, Hungary\\
% E-mails: biro.andras@renyi.mta.hu \,\,  lapkova.kostadinka@renyi.mta.hu}\end{center}

\begin{abstract}We solve unconditionally the class number one problem for the $2$-parameter family of real quadratic fields $\Q(\sqrt{d})$ with square-free discriminant
$d=(an)^2+4a$  for $a$ and
$n$ -- positive odd integers.
\end{abstract}



\maketitle

\section{Introduction}\label{sec:intro}

\hspace{0.5cm}Let us consider the quadratic fields $K=\Q(\sqrt{d})$ with class group $Cl(d)$ and order of the class group denoted by $h(d)$. In this paper we determine all fields $K=\Q(\sqrt{d})$ where $d=(an)^2+4a$ is square-free and $a$ and $n$ are positive odd integers such that the class number $h(d)$ is $1$. It follows from Siegel's theorem that there are only finitely many such fields, but since Siegel's theorem is ineffective, it cannot provide the specific fields with class number one. For this sake we apply the method developed by Bir\'o
in \cite{biro} and we use the result of Lapkova \cite{Lapkova}.\\

We remark that the class number one problem that we consider was
already suggested by Bir\'o in \cite{biroNagoya} as a possible
generalization of his works. The discriminants of the form $d=(an)^2 + ka$ for $\pm
k\in\{1,2,4\}$ are called Richaud-Degert type, so we consider here Richaud-Degert type discriminants with $k=4$. We expect that the same method will work for the other values of $k$ as well.

The class number one problem for special cases of
Richaud-Degert type was solved in
\cite{biro},\cite{biroChowla}, proving the Yokoi and Chowla Conjectures. The method was subsequently applied e.g in
\cite{byeon1} and \cite{LeeComplete}, but in these papers the parameter $a$ is fixed
($a=1$). However, already a subset of positive
density of the discriminants of Richaud-Degert type with $k=4$ are
covered in \cite{Lapkova}.\\

Under the assumption of a Generalized Riemann Hypothesis there is
a list of principal quadratic fields of Richaud-Degert type, see
\cite{MollinAll}, and one can check there that the largest number in that list having the form $d=(an)^2+4a$ is 1253. Here, however, our main result is unconditional:

\begin{thm}\label{thm:1}If $\,d=(an)^2+4a$ is square-free for $a$ and $n$ odd positive integers and $d>1253$, then
$h(d)>1$.
\end{thm}

\section{Notations and structure of the paper}

If $\chi$ is a Dirichlet character, then $L(s,\chi)$ denotes the usual Dirichlet $L$-function. If $d$ is a square-free positive integer and $d\equiv 1\pmod 4$, we denote by $\chi_d$  the real primitive Dirichlet character with conductor $d$, i.e. $\displaystyle\chi_d(m)=\Leg{m}{d}$ (Jacobi symbol).

$\Oc_K$ denotes the ring of integers of the quadratic field $K$. The norm $N\af$ of an integral ideal $\af$ in $\Oc_K$ is the index $[\Oc_K:\af]$. The Dedekind zeta function is defined as
\begin{equation}\zeta_K(s):=\sum_\af\frac{1}{(N\af)^s}
\end{equation}
where the summation is over all integral ideals $\af$ in $\Oc_K$. It is well-known (see e.g. Theorems 4.3 and 3.11 of \cite{wash}) that
\begin{equation}\label{eq:z20}\zeta_K(s)=\zeta(s)L(s,\chi_d)\,.
\end{equation}

 Throughout the paper by $(a,b)$ we denote the greatest common divisor of the integers $a$ and $b$ and $P^{+}(a)$ denotes the largest prime factor of $a$. As usual $\mu(x)$ means the M\"obius function. \\

 If $K$ is a real quadratic field, for $\beta\in K$ we denote its algebraic conjugate by $\overline{\beta}$. The element $\beta\in K$ is called \textit{totally positive}, denoted by $\beta\gg 0$, if $\beta>0$ and  $\bar{\beta}>0$. \\

%structure of the paper
The structure of the paper is the following: in \S\ref{sec:thm-b-g} we state the main result of \cite{biro-gr} on the evaluation of a partial zeta function in a general real quadratic field $K$, then we apply it for our special fields in \S\ref{sec:apply-biro}, and we derive there our main tool, Lemma \ref{main}. We simplify some quantities appearing in Lemma \ref{main} in \S\ref{sec:rem-formula}. We prove our main theorem in \S\ref{sec:proofThm}.\\

%about the computations
Computer calculations play an important role in the proof of the main theorem. These are SAGE (entry \cite{sage} from our bibliography) and C++ computations. The main number theoretic objects, characters, algebraic numbers and ideals in certain cyclotomic fields, are introduced in SAGE. The data obtained in SAGE we plug in programs (sieves) in C++ for speeding up the calculations, and most of the time we return to SAGE to finish our sieving with much less cases to consider and hence not bothering about the speed. The time for performing all possible computations was about $57$ hours, on an old personal laptop under Windows XP, with an AMD $64$x$2$ mobile processor at $1.6$ GHz speed, and $1$ GB RAM. All files can be found at \texttt{http://www.renyi.hu/$\sim$biroand/code/}(entry \cite{www} from the References of this paper) and more information about the implementation of the code is provided in the file \texttt{READ ME.txt} there.

\section{Bir\'o-Granville's Theorem}\label{sec:thm-b-g}
\hspace{0.5cm} In \cite{biro-gr} Bir\'o and Granville give a finite formula for a partial zeta function at $0$. They illustrate its efficiency with successful solving of the class number one problem for some one parameter R-D discriminants where $a=1$. Here we restate their main theorem.\\

Let $K$ be a real quadratic field with discriminant $d$, let $\chi$ be a Dirichlet character of conductor $q$ and let $I$ be an integral ideal of $K$. Define
$$\zeta_I(s,\chi):=\sum_\af\frac{\chi(N\af)}{(N\af)^s}
$$
where the summation is over all integral ideals $\af$ equivalent to $I$ in the ideal class group $Cl(d)$. For a quadratic form $f(x,y)\in\Z[x,y]$ introduce the sum
\begin{equation}\label{eq:G}
G(f,\chi):=\sum_{1\leq u,v\leq q-1}\chi\left(f(u,v)\right)\frac{u}{q}\frac{v}{q}\,.
\end{equation}
 According to the theory of cycles of reduced forms corresponding to a given ideal, see e.g. \S53 in \cite{hecke}, the ideal $I$ of $K$ has a $\Z$-basis $(\nu_1,\nu_2)$ for which $\nu_1\gg 0$ and $\alpha=\nu_2/\nu_1$ satisfies $0<\alpha<1$. Moreover, the regular continued fraction expansion of $\alpha$ is purely periodic:
$$\alpha=\left[0,\overline{a_1,\ldots,a_\ell}\right]$$
for some positive $\ell$ (which is the least period) and $a_1,\ldots,a_\ell$. Here $a_{j+\ell}=a_j$ for every $j\geq 1$. Further for $n\geq 1$ denote $$\frac{p_n}{q_n}=\left[0,a_1,\ldots,a_n\right]$$
and write $\alpha_n:=p_n-q_n\alpha$ with $\alpha_{-1}=1$ and $\alpha_0=-\alpha$. Define also for $j=1,2,\ldots$
$$Q_j(x,y)=\frac{1}{NI}\left(\nu_1\alpha_{j-1}x+\nu_1\alpha_jy\right)\left(\overline{\nu_1\alpha}_{j-1}x+\overline{\nu_1\alpha}_jy\right)$$
 and $$f_j(x,y)=(-1)^jQ_j(x,y).$$
It is known that every $f_j$ has integer coefficients. Using the usual notation
$$\tau(\chi):=\sum_{a(q)}\chi(a)e\left(\frac a q\right)$$
for the Gauss sum, introduce the expression
\begin{equation}\label{def:beta0}\beta_\chi:=\frac{1}{\pi^2}\chi(-1)\tau(\chi)^2L(2,\overline{\chi}^2).
\end{equation}
Also recall that a character $\chi$ is called \textit{odd} if $\chi(-1)=-1$. \\

In \cite{biro-gr} the following main result is proven
\begin{thm}[Bir\'o, Granville \cite{biro-gr}]\label{thm:biro-gr} Suppose that $\chi$ is an odd primitive character with conductor $q>1$ and $(q,2d)=1$. With the notations as above we have
$$\frac{1}{2}\zeta_I(0,\chi)=\sum_{j=1}^\ell G(f_j,\chi)+\frac{1}{2}\chi(d)\Leg{d}{q}\beta_\chi\sum_{j=1}^\ell a_j\overline{\chi}\left(f_j(1,0)\right).$$
\end{thm}

\section{Application of Theorem \ref{thm:biro-gr} for Our Special Discriminant}\label{sec:apply-biro}
 %%\section[Applying Theorem \ref{thm:biro-gr}]{Application of Theorem \ref{thm:biro-gr} for our Special Discriminant}\label{sec:apply-biro}
\hspace{0.5cm} Let $\,d=(an)^2+4a$ be square-free with odd positive integers $a$ and $n$ and assume that $a>1$. We use that $d\equiv 1\pmod 4$, so the ring of integers $\Oc_K$ of the field $K=\Q(\sqrt{d})$ is of the type $\displaystyle\Oc_K=\Z\left[1,(\sqrt{d}+1)/2\right]$.
Introduce
$$\displaystyle\alpha=\frac{\sqrt{d}-an}{2}\,.$$
We have $0<\alpha<1$ and we take the ideal $I=\Z[1,\alpha]$. Clearly $I=\Oc_K$ and we apply Theorem \ref{thm:biro-gr} to compute the partial zeta function for the class of principal ideals.\\

However to apply the upper formula for the function $\zeta_I$ we need the continued fraction expansion of $\alpha$. It can be checked by some computations, e.g. using \cite{schinzel} and the rules on page 78 from \cite{beck}, that
\begin{equation}\label{eq:beta/2}
\alpha=[0,\,\overline{n,an}\,].
\end{equation}
Using the notation from \S\ref{sec:thm-b-g} we have $\ell=2$, since we consider $a>1$,
and
\begin{equation}\label{eq:zeta-general}
\frac{1}{2}\zeta_I(0,\chi)=\sum_{j=1}^2 G(f_j,\chi)+\frac{1}{2}\chi(d)\Leg{d}{q}\beta_\chi\sum_{j=1}^2 a_j\overline{\chi}(f_j(1,0)).
\end{equation}
Here $p_1/q_1=[0;n]=1/n$, $p_2/q_2=1/(n+1/an)=an/(an^2+1)$ and $\alpha_1=1-n\alpha$, $\alpha_2=an-(an^2+1)\alpha$.\\

By the choice of the ideal $I=\Oc_K$ we have that $NI=1$ and $\nu_1=1$ and so
\begin{equation}\label{eq:Qj}
Q_j(x,y)=\alpha_{j-1}\overline{\alpha}_{j-1}x^2+(\alpha_{j-1}\overline{\alpha}_j+\alpha_j\overline{\alpha}_{j-1})xy+\alpha_{j}\overline{\alpha}_{j}y^2\,.
\end{equation}

Observe that $\alpha$ is the positive root of the equation $x^2+(an)x-a=0$. Then $\alpha+\bar\alpha=-an$ and $\alpha\bar\alpha=-a$. We use these to compute
\begin{eqnarray*}\label{eq:Q1}Q_1(x,y)& = &\alpha_0\bar\alpha_0 x^2+(\alpha_0\bar\alpha_1+\alpha_1\bar\alpha_0)xy+\alpha_1\bar\alpha_1 y^2\\
& = &\alpha\bar\alpha x^2+\left(-\alpha(1-n\bar\alpha)-\bar\alpha(1-n\alpha)\right)xy + (1-n\alpha)(1-n\bar\alpha)y^2\\
& = &-ax^2-anxy+y^2.
\end{eqnarray*}
Similarly
 \begin{eqnarray*}\label{eq:Q2}
  Q_2(x,y)&=&\alpha_1\bar\alpha_1 x^2+(\alpha_1\bar\alpha_2+\alpha_2\bar\alpha_1)xy+\alpha_2\bar\alpha_2 y^2  \\
  &=& (1-n\alpha)(1-n\bar\alpha)x^2\\
  & &+\left\{(1-n\alpha)\left(an-(an^2+1)\bar\alpha\right)
+(1-n\bar\alpha)\left(an-(an^2+1)\alpha\right)\right\}xy\\
  & & + \left(an-(an^2+1)\alpha\right)\left(an-(an^2+1)\bar\alpha\right)y^2 \\
 & = & x^2+anxy-ay^2.
 \end{eqnarray*}
 So
 \begin{equation}\label{eq:f1_2}
 f_1(x,y)=ax^2+anxy-y^2
 \end{equation}
and
 \begin{equation}\label{eq:f2_2}
 f_2(x,y)=x^2+anxy-ay^2\,.
\end{equation}
We see that $f_1(1,0)=a$ and $f_2(1,0)=1$. Introduce
\begin{equation}\label{def:c}c_a:=a+\overline{\chi}(a)\,.
\end{equation}
When we substitute in (\ref{eq:zeta-general}) we get
\begin{equation}\label{eq:zeta-1}
\frac{1}{2}\zeta_I(0,\chi)=G(f_1,\chi)+G(f_2,\chi)+\frac{n}{2}\chi(d)\Leg{d}{q}\beta_\chi c_a\,.
\end{equation}

Now assume that we are in a field $K$ where $h(d)=1$. Then all integral ideals are principal. So
\begin{equation}\label{eq:z-z}\zeta_I(s,\chi)=\sum_{\af\triangleleft\Oc_K}\frac{\chi(N\af)}{(N\af)^s}=:\zeta_K(s,\chi)\,.
\end{equation}
It follows easily from (\ref{eq:z20}) that
\begin{equation}\label{eq:z2}\zeta_K(s,\chi)=L(s,\chi)L(s,\chi\chi_d)\,.
\end{equation}

Recall (see e.g. Theorem 4.2 of \cite{wash}) the following
equation for an odd primitive character $\chi$:
\begin{equation}\label{eq:L0}L(0,\chi)=-\sum_{1\leq a\leq q}\chi(a)\frac{a}{q}\,.
\end{equation}

Let us further denote
\begin{equation}\label{def:m}
 m_\chi:=\sum_{1\leq a<q}a\chi(a)=-qL(0,\chi)\,.
 \end{equation}

Then from (\ref{eq:z-z}) and (\ref{eq:z2}) we have
$$q\zeta_I(0,\chi)=qL(0,\chi)L(0,\chi\chi_d)=-m_\chi L(0,\chi\chi_d)\,.$$
Combining the latter equality with (\ref{eq:zeta-1}) we get
\begin{equation}\label{eq:A1}-\frac{1}{2}m_\chi L(0,\chi\chi_d)=q\biggl(G(f_1,\chi)+G(f_2,\chi)+\frac{n}{2}\chi(d)\Leg{d}{q}\beta_\chi c_a)\biggr)\,.\end{equation}

Introduce the notation
\begin{equation}\label{def:c-a-n}C_\chi(a,n):=q\biggl(G(f_1,\chi)+G(f_2,\chi)\biggr)\,.
\end{equation}
Then (\ref{eq:A1}) transforms into
\begin{lem}\label{eq:A-C-n}With the upper notations, if $h(d)=1$, we have
$$ -m_\chi L(0,\chi\chi_d)=2C_\chi(a,n)+nq\chi(d)\Leg{d}{q}\beta_\chi c_a\,.$$
\end{lem}

Let $\Lf_\chi$ be the field formed by adjoining to $\Q$ all the
values of the character $\chi$ and $\Oc_{\Lf_\chi}$ be its ring of
integers. Note that $d\equiv 1\pmod 4$, so
$\displaystyle\Leg{-1}{d}=(-1)^{(d-1)/2}=1$ and $\chi_d$ is an
even character. Then we can state
\begin{cla}\label{factA} For the odd character $\chi$ with conductor $q$ and $d\equiv 1\pmod 4$ such that $(q,d)=1$ the quantity $L(0,\chi\chi_d)$ is an algebraic integer in the number field $\Lf_\chi$.
\end{cla}
This can be shown in the same way as the corresponding statement
above Fact A of \cite{biro}, using formula (\ref{eq:L0}) for the odd primitive
character $\chi\chi_d$ and the fact that $q$ and $d$ are coprime.

Take a prime ideal $\Rf$ in $\Oc_{\Lf_\chi}$ such that $m_\chi\in\Rf$. By Claim \ref{factA}
we have $L(0,\chi\chi_d)\in\Oc_{\Lf_\chi}$ so $-m_\chi
L(0,\chi\chi_d)\equiv 0\pmod{\Rf}$. Then by Lemma \ref{eq:A-C-n}
we get the main result of this section:

\begin{lem}\label{main}Let $\,d=(an)^2+4a$ be square-free with odd positive integers $a$ and $n$ and assume that $a>1$ and $h(d)=1$. Suppose that $\chi$ is an odd primitive character with conductor $q>1$ and $(q,2d)=1$. Take a prime ideal $\Rf$ in $\Oc_{\Lf_\chi}$ such that $m_\chi\in\Rf$. Then we have
\begin{equation}\label{eq:2A-0}0\equiv 2C_\chi(a,n)+n\chi(d)\Leg{d}{q}q\beta_\chi c_a\pmod{\Rf}\,.
\end{equation}
with the notations (\ref{eq:G}), (\ref{eq:f1_2}), (\ref{eq:f2_2}), (\ref{def:c-a-n}), (\ref{def:beta0}),
(\ref{def:c}).

\end{lem}

\section{Further Remarks on Lemma \ref{main}}\label{sec:rem-formula}
\hspace{0.5cm}First we find a more simple finite form for
$\beta_\chi$. Let
\begin{equation}\label{eq:gamma}\gamma_\chi:=\sum_{n=1}^{q-1}\chi^2(n)\frac{n^2}{q^2}
\end{equation}
and consider the Jacobi sum $$\displaystyle
J_\chi:=\sum_{\substack{a,b\pmod q\\a+b\equiv 1\pmod
q}}\chi(a)\chi(b)\,.$$ The following claim shows that $\beta_\chi$
is actually not only an algebraic number but also computable in
finitely many steps which is not at all evident from definition
(\ref{def:beta0}). The claim is stated in the Introduction of \cite{biro-gr} and it is proven in \S 6 of that paper.
 \begin{lem}\label{def:beta} Let $\chi$ be a primitive character of order greater than $2$. For the unique way to write $\chi =\chi_{+}\chi_{-}$ where
$\chi_{+},\,\chi_{-}$ are primitive characters of coprime
conductors $q_{+},\,q_{-}$ respectively, such that $\chi_-$ has
order $2$, and $\chi_{+}^2$ is also primitive, we have
$$\beta_{\chi}=\chi_{+}(-1)J_{\chi_{+}}\gamma_{\chi}\mu (q_{-})\prod_{p\mid q_{-}}\frac{p^2\chi_{+}^2(p)-1}{p\chi_{+}^2(p)-1}\,.$$
\end{lem}

The following statement is proved in \S9 of \cite{biro-gr}. As the
exposition in \cite{biro-gr} is somewhat sketchy we give here a
detailed proof.
\begin{lem}\label{cla:G1=G2}For odd complex character $\chi$ with conductor $q>2$ such that $(q,2d)=1$ we have $$G(f_1,\chi)=G(f_2,\chi)\,.$$
\end{lem}
\begin{proof} In (\ref{eq:G}) we change the summation by $u\rightarrow v\,,v\rightarrow q-u$. Then for the new variables again $1\leq v,q-u\leq q-1$. Now
\begin{eqnarray*}G(f_1,\chi)&=&\sum_{1\leq u,v\leq q-1}\chi(av^2+anv(q-u)-u^2)\frac{v}{q}\frac{q-u}{q}\\
&=&\sum_{1\leq u,v\leq
q-1}\chi(av^2-anvu-u^2)\frac{v}{q}\frac{-u}{q}
+\sum_{1\leq u,v\leq q-1}\chi(av^2-anvu-u^2)\frac{v}{q}\\
&=&\sum_{1\leq u,v\leq q-1}\chi(-1)\chi(-av^2+anvu+u^2)\frac{v}{q}\frac{-u}{q}-\sum_{1\leq u,v\leq q-1}\chi(f_2(u,v))\frac{v}{q}\\
&=&\sum_{1\leq u,v\leq
q-1}\chi(f_2(u,v))\frac{u}{q}\frac{v}{q}-\sum_{1\leq u,v\leq
q-1}\chi(f_2(u,v))\frac{v}{q}\,.
\end{eqnarray*}
We use the notation %from \cite{biro-gr}
\begin{equation}\label{def:g}
  g(\chi,f,h):=\sum_{1\leq m,n\leq q-1}\chi(f(m,n))h\left(\frac{n}{q}\right)
\end{equation}
for the quadratic form $f(x,y)=Ax^2+Bxy+Cy^2$ with square-free discriminant $\Delta=B^2-4AC$ and $h(x)\in\Z[x]$.\\

Therefore we have
\begin{equation*}\label{eq:G1-G2-g}G(f_1,\chi)=G(f_2,\chi)-g(\chi,f_2,t)\,.
\end{equation*}
We will prove that %as it is shown in $\S6$ of \cite{biro-gr} we have
\begin{equation}\label{eq:g_2}g(\chi,f_2,t)=0\,.
\end{equation}
We will make it by showing that $g(\chi,f_2,1)=0$ and $g(\chi,f_2,t-1/2)=0$.\\

First notice that there is a $\delta$ with $(\delta,q)=1$ such that $\chi(\delta)\neq 0,1$ and one can find $r,s$ for which $\delta\equiv r^2-\Delta s^2\pmod q$. The argument that follows is for square-free $q$ and the one for general $q$ follows easily. The existence of such $r$ and $s$ follows from the theory of norm residues modulo $q$ in $\Q(\sqrt{\Delta})$ for $(q,\Delta)=1$, see Theorem 138 and Lemma from \S47 in \cite{hecke}. Basically we use that the group of norm residues modulo $q$ is big, take element $\delta_1$ from it and then choose $\delta$ to be $\delta_1$ or $4\delta_1$ depending on the residue of the discriminant of the field modulo $4$. In this case $r^2-\Delta s^2$ is the norm, or four times the norm, of an algebraic integer in $\Q(\sqrt{\Delta})$. \\

Now if we choose $M$ and $N$ satisfying
$$(2AM+BN)+\sqrt{\Delta}N=\left((2Am+Bn)+\sqrt{\Delta}n\right)(r+\sqrt{\Delta}s)$$
we get
$$\left((2AM+BN)+\sqrt{\Delta}N\right)\left((2AM+BN)-\sqrt{\Delta}N\right)=4Af(M,N)=4Af(m,n)(r^2-\Delta s^2)\,.$$

>From definition (\ref{eq:f2_2}) the coefficient $A$ of $f_2$ equals $1$, i.e. $(A,q)=1$, so we get $f_2(M,N)\equiv f_2(m,n)\delta\pmod q$. One checks that
\begin{displaymath}
\left( \begin{array}{c}
M \\
N
\end{array}\right)=\left( \begin{array}{cc}
r-Bs & -2Cs \\
2As & r+Bs
\end{array}\right) \left( \begin{array}{c}
m \\
n
\end{array} \right)
\end{displaymath}
with determinant of the upper matrix, denoted by $\Tf$, equal to
$r^2-\Delta s^2\neq 0$. Since $\Tf$ is invertible and $m$ and $n$
are linear forms of $M$ and $N$, if some of the latter do not take
each residue modulo $q$ exactly $q$ times, then some of the
residues $m$ or $n$ will not either. Therefore when $0\leq m,n\leq
q-1$ also $0\leq M,N\pmod q\leq q-1$. Notice as well that
$$g(\chi,f,1)=\sum_{0\leq m,n\leq q-1}\chi(f(m,n))$$
because $\chi$ is not a real character and
$$\sum_{0\leq m\leq
q-1}\chi(Am^2)=\sum_{0\leq n\leq q-1}\chi(Cn^2)=0\,.$$ That is why
we can substitute $m$ and $n$ with $M$ and $N$ in the sum
$g(\chi,f_2,1)$. We get $g(\chi,f_2,1)=\chi(\delta)g(\chi,f_2,1)$.
Hence
\begin{equation}\label{eq:g0}g(\chi,f_2,1)=\sum_{1\leq m,n\leq q-1}\chi(f(m,n))=0\,.
\end{equation}
Further, consider the Bernoulli polynomial $\displaystyle B_1(x):=x-\frac{1}{2}$. We notice that $\displaystyle B_1(1-x)=\frac{1}{2}-x=-B_1(x)$. Therefore $\displaystyle\chi(f(m,n))B_1\left(\frac{n}{q}\right)=-\chi\left(f(q-m,q-n)\right)B_1\left(\frac{q-n}{q}\right)$ and  \begin{eqnarray*}g(\chi,f,B_1)&=&\sum_{1\leq m,n\leq q-1}\chi(f(m,n))B_1(\frac{n}{q})=-\sum_{1\leq m,n\leq q-1}\chi(f(q-m,q-n))B_1(\frac{q-n}{q})\\
&=&-g(\chi,f,B_1)\,.
\end{eqnarray*}
We got that $g(\chi,f,B_1)=0$. This and (\ref{eq:g0}) yield
(\ref{eq:g_2}) and therefore we complete the proof.
\end{proof}

Further we state
\begin{lem}\label{cla:C-C} For any odd character $\chi$ with conductor $q>2$ we have
$$C_\chi(a,q-n)=-C_\chi(a,n)\,.$$
\end{lem}
\begin{proof}To show this we substitute $n\rightarrow q-n$ in the definition of $G(f_1,\chi)$:
\begin{eqnarray*}G(f_1,\chi)_{q-n}&=&\sum_{1\leq x,y\leq q-1}\chi(ax^2+a(q-n)xy-y^2)\frac{x}{q}\frac{y}{q}\\
&=&\sum_{1\leq x,y\leq q-1}\chi(ax^2-anxy-y^2)\frac{x}{q}\frac{y}{q}\\
 &=&\sum_{1\leq x,y\leq q-1}\chi(-1)\chi(-ax^2+anxy+y^2)\frac{x}{q}\frac{y}{q}\\
 &=&-G(f_2,\chi)_n\,.
 \end{eqnarray*}
 Thus we have that
 $$\frac 1 q C_\chi(a,q-n)=G(f_1,\chi)_{q-n}+G(f_2,\chi)_{q-n}=-G(f_2,\chi)_n-G(f_1,\chi)_n=-\frac 1 q C_\chi(a,n)\,.$$\end{proof}

As an immediate corollary we also get
\begin{lem}\label{cla:C0} For any odd character $\chi$ with conductor $q>2$ and for any integer $a$ we have $$C_\chi(a,0)=0\,.$$
\end{lem}
Indeed, $C_\chi(a,0)=C_\chi(a,q-0)=-C_\chi(a,0)$ and therefore the claim. This also means that under the conditions of Lemma \ref{cla:G1=G2} for any $n$ divisible by $q$ we have $C_\chi(a,n)=0$ and therefore $G(f_1,\chi)=0$ as well.\\

\section{Proof of Theorem \ref{thm:1}}\label{sec:proofThm}
%%\section[On the proof of Theorem \ref{thm:BGL}]{On the Proof of Theorem \ref{thm:BGL} and Further Plans}\label{sec:proofThm}
\hspace{0.5cm}Let $d$ be as in Theorem \ref{thm:1}. We assume in the sequel that $a>1$, since the case $a=1$ follows from Yokoi's conjecture proved in \cite{biro}.

Suppose now that $\chi$ is an odd primitive
character modulo $q>1$ and $(q,2d)=1$.  Assume, in addition, that
$\chi$ is a complex character, i.e. $\chi^2\neq 1$.

In this case below we will use Lemma \ref{main}, Lemma \ref{cla:G1=G2} and Lemma
\ref{def:beta}. By (\ref{def:c-a-n}) and (\ref{eq:2A-0}) we get
\begin{eqnarray}\label{eq:congrB}4q^2\left(\prod_{\left.p\right|q^{-}}\left(p\chi_{+}^2(p)-1\right
)\right)G\left(f_1,\chi\right)& + & \nonumber\\
+n\chi (d)\left({\frac d q}\right)c_aq^2J_{\chi_{+}}\gamma_{\chi}
\mu (q_{-})\chi_{+}(-1)\left(\prod_{p\mid q^{-}}\left(p^2
\chi_{+}^2(p)-1\right)\right)&\equiv& 0\pmod\Rf\,,
\end{eqnarray}
where the prime ideal $\Rf$ of $\Lf_\chi$ lies above the rational prime
$r$, we suppose $m_\chi\in\Rf$ and $(r,q)=1$. Then it is clear, using
(\ref{eq:G}), the definition of $f_1$ and $c_a$ in (\ref{eq:f1_2}) and
(\ref{def:c}), that the truth of (\ref{eq:congrB})
depends only on the residues of $a$ and $n$ modulo $qr$.\\

Let us now define a directed graph in a similar but slightly
different way than in \cite{biro}. Let us denote by an arrow
$$q\rightarrow r$$
that the following conditions are true: $q>1$ is an odd integer,
there is an odd primitive character $\chi$ modulo $q$ such that
$\chi^2\neq 1$, and there is a prime ideal $\Rf$ of $\Lf_{\chi}$
such that $\Rf$ lies above the odd rational prime $r$, which
satisfies $(r,q)=1$ and $m_\chi\in\Rf$. The latter condition can
arise for example for an odd character if $r\mid h_q^{-}$, where
$h_q^{-}$ is the relative class number
of the cyclotomic field $\Q(\zeta_q)$ for $\zeta_q=e\br{1/q}$ (Theorem 4.17 \cite{wash}).\\

We will use the following claim which was proved as Claim 5.1 of
\cite{Lapkova} as a generalization of Fact B of \cite{biro}.
\begin{cla}\label{factB} If $h(d)=1$ for the square-free discriminant $d=(an)^2+4a$, then $a$ and $an^2+4$ are primes, and for any prime $p\neq a$ such that $2<p<an/2$ we have $$\Leg{d}{p}=-1\,.$$
\end{cla}
Also we recall the statement of Theorem 1.1 of \cite{Lapkova}.
\begin{thm}\label{thm:La}If $d = (an)^2 + 4a$ is square-free for odd positive integers $a$ and $n$ such that
$43 \cdot 181 \cdot 353 \mid n$, then $h(d) > 1$.
\end{thm}

Let $q\rightarrow r$ hold. Then by the considerations above and by
Claim \ref{factB} we get that if $h(d)=1$ for the square-free
discriminant $d=(an)^2+4a$ satisfying $P^{+}(qr)<an/2$, and $a$ is
different from any prime factor of $qr$, then
\begin{equation}\label{eq:Leg_m1}\Leg{(an)^2+4a}{p}=-1
\end{equation}
for every prime divisor $p$ of $qr$, and (\ref{eq:congrB}) also
holds. We see that (\ref{eq:Leg_m1}), similarly to
(\ref{eq:congrB}), depends only on the residues of
$a$ and $n$ modulo $qr$.\\

\begin{lem}\label{thm:BGL} If $d=(an)^2+4a$ is square-free for odd positive integers $a$ and
$n$ with $an>2\cdot 127$,
\begin{equation}a \neq 1,3,5,7,13,17,19,37,73,127\,
\end{equation}
and $h(d)=1$, then we have
$$n\equiv 0 \pmod {3\cdot 5\cdot 7\cdot 13\cdot
19\cdot 37}\,.$$
\end{lem}

\begin{proof} We apply the arrows
$$5\times 19\rightarrow 13,$$
$$7\times 19\rightarrow 13,37,73,$$
$$13\times 19\rightarrow 3,7,73,127,$$
$$3\times 5\times 19\rightarrow 37,73,$$
$$7\times 13\rightarrow 37,$$
$$3\times 73\rightarrow 17,$$
$$3\times 37\rightarrow 19,$$
$$5\times 37\rightarrow 13,$$
$$3\times 7\times 13\rightarrow 19,37,$$
$$7\times 17\rightarrow 5,$$
$$127\rightarrow 5,13,$$
$$3\times 127\rightarrow 37 .$$

It is easy to check that the maximal prime factor of any $q$ is at most $127$ and the maximal value of $r$ is $127$, so our
conditions guarantee that $P^{+}(qr)<an/2$, and $a$ is different from any
prime factor of $qr$ in every case. One can check by concrete
computations (finding a suitable character and a suitable prime
ideal in every case) that these are indeed arrows.\\

Let $P:=3\cdot 5\cdot 7\cdot 13\cdot 19\cdot 37$, and let us denote by $A$ the set of those arrows from the above list where $qr$ consists only of primes dividing $P$. Let us denote by $B$ the set of those arrows from the above list which are not in $A$, i.e. where $qr$ is divisible by $17$, $73$ or $127$.

In the first part of the proof we apply only the arrows from  $A$.  We fix the residue $a_0$ of $a$ and $n_0$ of $n$ modulo $P$, and then the residues of $a$ and $n$ modulo $qr$ are determined for every arrow from $A$. For every fixed pair $0 \le a_0,n_0 < P$  we check (\ref{eq:congrB}) and (\ref{eq:Leg_m1}) for every such arrow. We find that for most pairs $(a_0,n_0)$ the implied conditions yield $n_0=0$. In the second part of the proof it is enough to deal with the exceptional $(a_0,n_0)$ pairs, i.e with those pairs for which $n_0>0$ and (\ref{eq:congrB}) and (\ref{eq:Leg_m1}) are true for this pair and for every arrow from  $A$.

In the second part of the proof we increase the modulus to $P\cdot 17\cdot 73\cdot 127$. We fix the residues $A_0$ of $a$ and $N_0$ of $n$ modulo $P\cdot 17\cdot 73\cdot 127$, but we consider only such pairs $0 \le A_0,N_0 < P\cdot 17\cdot 73\cdot 127$ for which there is an exceptional pair $(a_0,n_0)$ in the above sense such that $A_0 \equiv a_0\pmod {P}$ and $N_0 \equiv n_0\pmod {P}$. For every such pair $(A_0,N_0)$ and for every arrow from $B$ we check (\ref{eq:congrB}) and (\ref{eq:Leg_m1}). This eventually leads only to cases $N_0=0$, which implies $n_0=0$. This proves the lemma.

We explained in this way the theoretical part of the proof, but
the computer calculations are also very important. To save space
we do not present them here, but one can find them at the address \cite{www}.

\end{proof}

In the sequel we will use such cases when $q\rightarrow r$ holds, $h(d)=1$ for the square-free
discriminant $d=(an)^2+4a$ satisfying $P^{+}(qr)<an/2$, $a$ is
different from any prime factor of $qr$ (just as above), and in addition, either $r$ divides $n$, or $q$ divides $n$. Note that in the first case we have that $n \in \Rf$ (since $\Rf$ lies above $r$), so (\ref{eq:congrB}) reduces to
\begin{eqnarray}\label{eq:congrC}4q^2\left(\prod_{\left.p\right|q^{-}}\left(p\chi_{+}^2(p)-1\right
)\right)G\left(f_1,\chi\right)&
\equiv& 0\pmod\Rf\,,
\end{eqnarray}
so in this case (\ref{eq:Leg_m1}) and (\ref{eq:congrC}) are valid.\\

If $q$ divides $n$, from Lemma \ref{cla:C0} we get $G(f_1,\chi)=0$, so (\ref{eq:congrB}) transforms to
\begin{equation}\label{eq:congrE} n\chi (d)\left({\frac d q}\right)c_aq^2J_{\chi_{+}}\gamma_{\chi}
\mu (q_{-})\chi_{+}(-1)\left(\prod_{p\mid q^{-}}\left(p^2
\chi_{+}^2(p)-1\right)\right)\equiv 0\pmod\Rf\,.
\end{equation}
We remark that most of the factors in this congruence are easily checked to be nonzero modulo $\Rf$ (this can be computed for any particular parameters $q$ and $r$), so in practice the only remaining condition will be
\begin{equation*}\label{eq:congrD} c_a\equiv 0\pmod\Rf\,,
\end{equation*}
but we will check (\ref{eq:congrE}) itself in every case.

The proofs of the next three lemmas are very similar to each other. They are also similar to the proof of the previous lemma, but this time we will check (\ref{eq:Leg_m1}) and (\ref{eq:congrC}), or (\ref{eq:Leg_m1}) and (\ref{eq:congrE}).

\begin{lem}\label{thm:BGL2} If $d=(an)^2+4a$ is square-free for odd positive integers $a$ and
$n$ with $an>2\cdot 43$,
\begin{equation}a \neq 1,5,7,19,37,43\,,
\end{equation}
$n\equiv 0 \pmod {5\cdot 7\cdot 19\cdot 37}\,$ and $h(d)=1$, then we
have
$$n\equiv 0 \pmod {43}\,.$$

\end{lem}

\begin{proof} We apply the arrows
$$5\times 43\rightarrow 7,19,37.$$
One can check again by concrete
computations (finding a suitable character and a suitable prime
ideal in every case) that these are indeed arrows. By our considerations above we know that (\ref{eq:Leg_m1}) and (\ref{eq:congrC}) must be valid because for these three arrows $r$ divides $n$.

 We fix the residue $a_0$ of $a$ and $n_0$ of $n$ modulo $P:=5\cdot 7\cdot
19\cdot 37\cdot 43\,,$ but we consider only such cases when $n_0\equiv 0 \pmod {5\cdot 7\cdot 19\cdot 37}\,.$ For every such fixed pair $0 \le a_0,n_0 < P$ for which $n_0$ satisfies the above congruence we check (\ref{eq:Leg_m1}) and (\ref{eq:congrC}) for each arrows listed
above. We find that if the pair $(a_0,n_0)$ is such that $n_0
> 0$ is true, then either (\ref{eq:Leg_m1}) or (\ref{eq:congrC}) will be false for at
least one arrow. The necessary computer calculations can be found at \cite{www}. The lemma is proved.

\end{proof}

\begin{lem}\label{thm:BGL3} If $d=(an)^2+4a$ is square-free for odd positive integers $a$ and
$n$ with $an>2\cdot 181$,
\begin{equation}a \neq 1,3,5,13,19,37,181\,,
\end{equation}
$n\equiv 0 \pmod {3\cdot 5\cdot 13\cdot 19\cdot 37}\,$ and $h(d)=1$, then we
have
$$n\equiv 0 \pmod {181}\,.$$

\end{lem}

\begin{proof} We apply the arrows
$$181\rightarrow 5,37,$$
$$13\times 19\rightarrow 181,$$
$$3\times 5\times 19\rightarrow 181.$$

One can check again by concrete
computations (finding a suitable character and a suitable prime
ideal in every case) that these are indeed arrows. %By our considerations above we know that (\ref{eq:Leg_m1}) and (\ref{eq:congrC}) must be valid.

 We fix the residue $a_0$ of $a$ and $n_0$ of $n$ modulo $P:=3\cdot 5\cdot
13\cdot 19\cdot 37\cdot 181\,,$ but we consider only such cases when $n_0\equiv 0 \pmod {3\cdot 5\cdot 13\cdot 19\cdot 37}\,.$ For every such fixed pair $0 \le a_0,n_0 < P$ for which $n_0$ satisfies the above congruence we check (\ref{eq:Leg_m1}) and (\ref{eq:congrC}) for the first two arrows $181\rightarrow 5,37$ (here $r$ divides $n$). For the remaining pairs with $n_0>0$ we check (\ref{eq:Leg_m1}) and (\ref{eq:congrE}) ($q$ divides $n$). We find that if the pair $(a_0,n_0)$ is such that $n_0> 0$ is true, then either (\ref{eq:Leg_m1}) or (\ref{eq:congrE}) will be false for at
least one arrow. The necessary computer calculations can be found at \cite{www}. The lemma is proved.

\end{proof}

\begin{lem}\label{thm:BGL4} If $d=(an)^2+4a$ is square-free for odd positive integers $a$ and
$n$ with $an>2\cdot 353$,
\begin{equation}a \neq 1,3,5,13,17,353\,,
\end{equation}
$n\equiv 0 \pmod {3\cdot 5\cdot 13\cdot 17}\,$ and $h(d)=1$, then we
have
$$n\equiv 0 \pmod {353}\,.$$

\end{lem}

\begin{proof} We apply the arrows
$$3\times 5\times 17\rightarrow 353,$$
$$3\times 5\times 13\times 17\rightarrow 353.$$

One can check again by concrete
computations (finding a suitable character and a suitable prime
ideal in every case) that these are indeed arrows. By our considerations above we know that (\ref{eq:Leg_m1}) and (\ref{eq:congrE}) must be valid.

 We fix the
residue $a_0$ of $a$ and $n_0$ of $n$ modulo $P:=3\cdot 5\cdot
13\cdot 17\cdot 353\,,$ but we consider only such cases when $n_0\equiv 0 \pmod {3\cdot 5\cdot 13\cdot 17}\,.$ For every such fixed pair $0 \le a_0,n_0 < P$ for which $n_0$ satisfies the above congruence we check (\ref{eq:Leg_m1}) and (\ref{eq:congrE}) for each arrows listed
above. We find that if the pair $(a_0,n_0)$ is such that $n_0
> 0$ is true, then either (\ref{eq:Leg_m1}) or (\ref{eq:congrE}) will be false for at
least one arrow. The necessary computer calculations can be found at \cite{www}. The lemma is proved.

\end{proof}

We now prove the theorem assuming that $an>2\cdot 353$ and
\begin{equation}a \neq 3,5,7,13,17,19,37,43,73,127,181,353\,.
\end{equation}
Assume $h(d)=1$, then $an>2\cdot 17$ and $a\neq 3,5,7,13,17$ follows from above. Similarly like before for fixed residues $a_0$ of $a$ and $n_0$ of $n$ modulo $P:=3\cdot 5\cdot 7\cdot 13\cdot 17$ we check the conditions (\ref{eq:Leg_m1}) and (\ref{eq:congrC}) for the arrows

$$7\times 17\rightarrow 3,5,13,$$
$$13\times 17\rightarrow 5.$$

We find that if the pair $(a_0,n_0)$ is such that $n_0
> 0$ is true, then either (\ref{eq:Leg_m1}) or (\ref{eq:congrC}) will be false for at
least one arrow. The necessary computer calculations can be found at the address \cite{www}. We get in this way that $17$ divides $n$.\\

Let us also apply Lemma \ref{thm:BGL}. It follows that the conditions of Lemmas \ref{thm:BGL2}, \ref{thm:BGL3} and \ref{thm:BGL4} are satisfied. Then applying these lemmas it follows that $n\equiv 0 \pmod {43\cdot 181\cdot 353}\,$. This contradicts Theorem \ref{thm:La}. Hence our theorem is proved assuming the above two conditions. Since the finitely many cases $an \le 2\cdot 353$ are easily checked (the computations can be found at \cite{www}), it is enough to prove the theorem if $a$ equals one of the values
\begin{equation}3,5,7,13,17,19,37,43,73,127,181,353\,.
\end{equation}
This means that we almost finished the proof, since we reduced our original two-parameter problem to finitely many one-parameter problems. To complete the proof we will prove the theorem for these finitely many values of $a$.

For most of the exceptional cases we can apply exactly the same arrows as in \cite{biro}, for the case of  Yokoi's Conjecture, i.e. for $a=1$. Indeed, for
\begin{equation}\label{eq:excep}a=3,13,17,19,37,43,73,127,181,353\,
\end{equation}
we use the arrows
$$175\rightarrow 1861,61,$$
$$61\rightarrow 1861,$$
$$61\rightarrow 41.$$
We fix the residue $n_0$ of $n$ modulo $P:=41\cdot 61\cdot
175\cdot 1861\,$. For every fixed pair $(a,n_0)$, where $a$ is one of the values given in (\ref{eq:excep}) and $0 \le n_0 < P$ we check (\ref{eq:congrB}) and (\ref{eq:Leg_m1}) for every arrow given above. We find that for every such pair $(a,n_0)$ we get a contradiction for at least one arrow. This proves the theorem for the values in (\ref{eq:excep}) for the case $1861<an/2$. For smaller values of $n$ we can check the statement directly. The details of the computations can be found again at \cite{www}.

It remains to consider the cases $a=5$ and $a=7$.

For $a=5$ we use the arrows
$$61\rightarrow 1861,$$
$$61\rightarrow 41,$$
$$41\rightarrow 11.$$
We fix the residue $n_0$ of $n$ modulo $P:=11\cdot 41\cdot
61\cdot 1861\,$. For $a=5$ and for every fixed $0 \le n_0 < P$ we check (\ref{eq:congrB}) and (\ref{eq:Leg_m1}) for every arrow given above. We find that for every such $n_0$ we get a contradiction for at least one arrow. This proves the theorem for $a=5$ for the case $1861<5n/2$. For smaller values of $n$ we can check the statement directly. The details of the computations can be found at \cite{www}.

For $a=7$ we use the arrows
$$61\rightarrow 1861,$$
$$61\rightarrow 41,$$
$$41\rightarrow 11,$$
$$11,19\rightarrow 61,$$
$$9\rightarrow 11,$$

We fix the residue $n_0$ of $n$ modulo $P:=9\cdot 11\cdot 19\cdot 41\cdot
61\cdot 1861\,$. For $a=7$ and for every fixed $0 \le n_0 < P$ we check (\ref{eq:congrB}) and (\ref{eq:Leg_m1}) for every arrow given above. We find that for every such $n_0$ we get a contradiction for at least one arrow. This proves the theorem for $a=7$ for the case $1861<7n/2$. For smaller values of $n$ we can check the statement directly. The details of the computations can be found at \cite{www}.

The theorem is proved.

\subsection*{Acknowledgments.} We would like to thank L. Washington, R. Schoof and T. Mets\"ankyl\"a for the helpful correspondence which led to finding the arrow $ 3315\rightarrow 353$. This is the only arrow which was not suggested by the table for relative class numbers in Washington's book \cite{wash}.\\

The first author is partially supported by the Hungarian National Foundation for Scientific Research (OTKA) Grants  no. K100291, K104183, K109789 and ERC-AdG. Grant no. 321104. The second author is supported by Back-to-Research Grant of University of Vienna and partially supported by OTKA no. K104183.

\end{document}